\documentclass[11pt, a4paper]{amsart}

\usepackage{graphicx,amsmath,amssymb}
\usepackage{amsthm, mathrsfs}
\usepackage{enumerate}
\usepackage{enumitem}
\usepackage[disable]{todonotes}
\usepackage[all]{xy,xypic}
\usepackage{url}

\newtheorem{Theorem}{Theorem}
\newtheorem{Lemma}[Theorem]{Lemma}
\newtheorem{Fact}[Theorem]{Fact}
\newtheorem{Proposition}[Theorem]{Proposition}

\newtheorem{Remark}[Theorem]{Remark}
\theoremstyle{definition}
\newtheorem{Definition}[Theorem]{Definition}

\newtheorem{Example}[Theorem]{Example}

\theoremstyle{remark}

\newcommand{\abs}[1]{\left\vert#1\right\vert}
\newcommand{\satisfies}{\models}
\renewcommand{\bar}{\overline}
\renewcommand{\implies}{\rightarrow}
\renewcommand{\phi}{\varphi}

\newcommand{\ideal}{\triangleleft}
\newcommand{\clsubset}{\subseteq_{cl}}
\newcommand{\opsubset}{\subseteq_{op}}
\newcommand{\pr}{\operatorname{pr}}

\newcommand{\Char}[1]{\operatorname{Char}(#1)}

\newcommand{\Dom}[1]{\operatorname{Dom}(#1)}
\newcommand{\Ker}[1]{\operatorname{Ker}(#1)}

\newcommand{\Aut}[1]{\operatorname{Aut}(#1)}

\title{Specialisations and Algebraically Closed Fields}
\author{U\u gur Efem}
\address{University of Oxford, Mathematical Institute, Oxford, UK}
\email{efem@maths.ox.ac.uk}

\begin{document}
\begin{abstract}
We study the theory specialisations in algebraic geometry from a model theoretic 
viewpoint. In particular we investigate universality and maximality of specialisations 
in algebraic geometry.
\end{abstract}

\maketitle

\section{Introduction}
\label{Sec:1}

Specialisations in algebraic geometry were introduced early during 
the development of the subject (e.g. \cite{Weil}), and they are well 
understood. In abstract algebraic geometry over an algebraically 
closed field of an arbitrary characteristic, specialisations replace the 
topological tools coming from the analytic structure of the field $\mathbb C$. 
In other words, specialisations are the algebraic analogue of taking 
limits, or evaluating a continuous function at a point. 

In this paper we work in the setting of Zariski structures, a model theoretic 
generalisation of algebraic geometry (and of algebraically closed fields in 
particular) introduced by Hrushovski and Zilber~\cite{HZ}. Loosely speaking, 
a Zariski structure is a first order structure $(M_0,\tau)$ where $\tau$ is a 
family of topologies consisting of a topology for each $M_0^n$ which satisfies 
certain topological axioms arising from the Zariski topology of an algebraic variety. 
From a model theoretic perspective, elements of $\tau$ are the basic relations 
of the first order structure $(M_0,\tau)$. Further, we assume a good notion of 
dimension for definable sets exists. 

In~\cite{HZ} Hrushovksi and Zilber have shown that for one dimensional Zariski 
geometries, specialisations play an important role. It is also well known that 
for a Zariski structure $(M_0,\tau)$, any universal specialisation of $M_0$ 
can replace the collection of topologies $\tau$. In a preprint~\cite{OZ}, Onshuus 
and Zilber initiated a study of universal specialisations of Zariski structures. 
Where, they do not restrict to one dimensional case. This paper is motivated 
by their preprint, but aims to start a different route to study specialisations, 
and especially universal specialisations.

In this paper, we study the specialisations of algebraically closed fields, 
and varieties definable in them. For fields, the theory of specialisations 
is given by valuations. We show that the theory of specialisations of Zariski 
geometries arising from algebraic geometry is also essentially a part of 
a theory of algebraically closed valued fields.

Valuation theory is well studied both algebraically and model theoretically. 
In comparison, a general theory of specialisations of Zariski structures is 
not as deeply studied. In this paper we take a step forward from valuations 
of fields and as mentioned study specialisations of algebraic varieties. As it 
turns out, this is an important first step in the study of general theory of 
specialisations. 

In Section~\ref{Sec:prelim} we introduce Zariski structures and specialisations 
in detail, in the model theoretic context; and recall some basic constructions 
and facts.

Section~\ref{sct:acvf} considers specialisations of algebraically closed fields. 
We show that they are essentially residue maps of valuations (Theorem~\ref{thm:ext-res}). 
In this section we also give a first order theory for specialisations of algebraically 
closed valued fields, whose saturated models have only universal specialisations 
(Lemma~\ref{lem:saturated-universal}). The theory in question is a theory of 
algebraically closed valued fields, and it is very close other theories of acvf 
present in the model theory literature such as~\cite{Delon2012, HHM, HHM2, 
HK, Leloup, Robinson}. We also prove that our theory admits quantifier elimination, 
and is complete (Lemma~\ref{lem:QE-complete}). Here we should note that 
the proof for quantifier elimination is very to close to proofs of quantifier 
elimination results in~\cite[Lemma 6.3]{HK} and in~\cite{Leloup} for similar 
theories of algebraically closed valued fields. As a result of completeness, 
any model of our theory is elementarily equivalent to a model where the 
specialisation is universal (Theorem~\ref{thm:res-el-equiv}). 

In Section~\ref{sct:varieties}, as the natural next step, we consider specialisations 
of varieties definable in algebraically closed fields. Here we explain their 
relation to the specialisations of the algebraically closed fields they are 
definable in. We show that the universality of a specialisation of an affine 
variety is determined by the universality a corresponding specialisation of 
the field it is defined over (Proposition~\ref{prop:affine-univ-1}). We also 
give a characterisation of specialisations of varieties in term of the specialisations 
of the fields they are defined over (Propositions~\ref{prop:characterization-spcl-affine-varieties} 
and~\ref{prop:characterization-spcl-qp-varieties}).

As we have already mentioned specialisations of the algebraically 
closed field in question are residue maps of valuations, this can be seen 
as and extension of the valuations to varieties. Through this perspective, 
one can also say that specialisations of algebraic varieties are essentially 
a part of the theory of algebraically closed valued fields.

\section{Preliminaries}\label{Sec:prelim}
\subsection{Zariski Structures}
Let $\mathcal L$ be a first order language, and $M$ be an $\mathcal L$-structure. 
The closed subsets of $M^n$ for all $n\in \mathbb N$ are defined as 
follows: The subsets defined by primitive relations of $\mathcal L$ are 
(definably) \emph{closed}, and moreover the sets which are given by 
positive quantifier free $\mathcal L$-formulas are \emph{closed}.
\newline The topological axioms that distinguish the closed  sets can be 
given as:
\begin{enumerate}
\item Finite intersection of closed sets is closed;
\item Finite unions of closed sets are closed;
\item $M$ is closed;
\item The graph of equality is closed;
\item Any singleton in $M$ is closed;
\item Cartesian products of closed sets are closed;
\item The image of a closed set under a permutation of coordinates is closed;
\item For $a\in M^k$ and closed $S\subseteq M^k+l$, the set $S(a,M^l)$ 
is closed.
\end{enumerate}

A \emph{constructible sets} is defined to be a boolean combination of 
closed sets. In other words they are nothing but quantifier free 
$\mathcal L$-formulas. A set $S$ is said to be \emph{irreducible} if there 
are no proper relatively closed subsets $S_1,S_2\subset S$ such that $S = 
S_1 \cup S_2$. A topological structure is called \emph{Noetherian} if it 
satisfies the descending chain condition (DCC) for closed sets. 

A topological structure is said to be semi-proper if it satisfies the following 
condition:
\begin{description}
\item[(SP)]\label{semi-prop} Semi Properness: For a closed irreducible 
$S\subset M^n$ and a projection $\pr:M^n\to M^m$, there is a proper 
closed subset $F\subset \bar{\pr(S)}$ such that $\bar{\pr(S)}\setminus 
F \subset \pr(S)$  
\end{description}

Remark that by Noetherianity any closed set $S$ can be written as $S= 
S_1\cup\ldots\cup S_k$ where $S_1,\ldots, S_k\subset S$ are distinct and 
relatively closed; moreover they are unique up to ordering. They are called 
\emph{irreducible components} of $S$.

The \emph{dimension}, denoted by $\dim$ is a function from definable 
sets of  Noetherian topological structure to natural numbers which satisfies:
\begin{description}
\item[(DP)] Dimension of a Point: $\dim(a) = 0$ for all $a\in M$;
\item[(DU)] Dimension of Unions: $\dim(S_1\cup S_2) = \max(\dim(S_1), 
\dim(S_2))$ for closed $S_1$, and $S_2$;
\item[(SI)] For any $S\clsubset U\opsubset M^n$ and any closed $S_1
\subsetneq S$, $\dim(S_1)<\dim(S)$;
\item[(AF)] Addition Formula: For any irreducible closed $S\clsubset U
\opsubset M^n$ and a projection $\pr:M^n\to M^m$,
\[\dim(S) = \dim(\pr(S)) + \min_{a\in\pr(S)}(\pr^{-1}(a)\cap S) \]
\item[(FC)] Fiber Condition: Given $S\clsubset U\opsubset M^n$ and a 
projection $\pr:M^n\to M^m$, there is a relatively open $V\opsubset\pr(S)$ 
such that, for any $v\in V$
\[\min_{a\in\pr(S)}(\dim(\pr^{-1}(a)\cap S)) = \dim(\pr^{-1}(v)\cap S)\] 
\end{description} 

A Noetherian, semi-proper topological structure together with a dimension 
function which satisfies the above conditions is said to be a \emph{Noetherian Zariski Structure}.

From the definition one can easily observe that a constructible set $Q$ 
can be written as 
\[Q = \bigcup_{i\leq k} S_i\setminus P_i\]
for some $k$, closed $S_i,P_i$ such that $P_i\subset S_i$ and $S_i$ irreducible. 
Therefore clearly,
\[\bar{Q}=\bigcup_{i\leq k}S_i\]

An important consequence of the axioms is that a Zariski Structure $M$ has 
quantifier elimination; in other words, every definable set is constructible. 

\begin{Fact}[Theorem 3.2.1 in~\cite{Zilber}]\label{thm:zar-qe}
A Zariski Structure $M$ has quantifier elimination. I.e. every definable set 
is constructible. 
\end{Fact}

The axiom (SP) is an essential ingredient in the proof (see~\cite{Zilber}, 
Theorem 3.2.1). In fact, Sustretov later observed that (SP) is equivalent 
to quantifier elimination~\cite[Lemma 2.1.7]{Sustretov}.

By quantifier elimination, and the above observation one can extend 
the dimension to definable sets as follows:
\[\dim(Q) = \dim(\bar{Q}) = \max_{i\leq k}\dim(S_i)\]


\begin{Example}\label{ex:alg-closed}
Let $K$ be an algebraically closed field. Consider $K$ in the natural 
language for algebraic varieties, i.e. the primitive relations of the language 
are Zariski closed subsets of $K^n$ for $n\in\mathbb N$.

\begin{Theorem}
An algebraically closed field $K$ together with the natural language 
for varieties and the dimension is taken as the Krull dimension is a 
Zariski Structure. 
\end{Theorem}

The topological axioms 1-8 and Noetherianity are immediate from the 
definition of Zariski structures. The rest of the axioms are satisfied 
due to algebraic geometric results, which can be found in the literature. 
For some references see Theorem 3.4.1 in \cite{Zilber}.
\end{Example}

%
%
%

\subsection{Specialisations}\label{sec:spcl}

\begin{Definition}\label{def:specialization}
Let $\pi$ be a (partial) function $\pi:M\to M_0$ where $M_0$ is a Noetherian Zariski Structure, 
$M$ an elementary extension of $M_0$ and  for every formula $S(\bar{x})$ over $\emptyset$, defining 
an $M_0$-closed set and for every $\bar{a}\in M^n$ (which is also in the domain of $\pi$), $M\satisfies S(\bar{a})$ 
implies $M_0\satisfies S(\pi \bar{a})$. Such a function is said to be a \emph{specialisation}. We will 
also call the tuple $(M,\pi)$ a \emph{specialisation}.

A specialisation is said to be \emph{$\kappa$-universal} if, given any $M^\prime\succeq M\succeq M_0$, 
any $A\subseteq M^\prime$ with $|A|<\kappa$ and a specialisation $\pi_A:M\cup A\to M_0$ extending 
$\pi$, there is an embedding $\sigma:A\to M$ over $M\cap A$ such that $\pi_A|A = \pi\circ\sigma$.

An $\omega$-universal specialisation is said to be \emph{universal}.
\end{Definition}

In other words specializations are $\mathcal L$-homomorphisms from substructures of $M$ to $M_0$. 
The following lemma, originally due van den Dries \cite{vdD} and proved more generally for first order 
structures, gives a description of $M_0$-closed sets in terms of specializations.

\begin{Lemma}
 For each model $M$ of $Th(M_0)$, and each specialization $\pi:A\to M_0$ with $A\subseteq M$, if $T(x)$ is
 and $\mathcal L$ - formula such that $M\models T(a)$ implies $M_0\models T(\pi(a))$ then, there is a positive 
 quantifier free formula $S(x)$ such that $Th(M_0)\vdash T(x)\longleftrightarrow S(x).$
\end{Lemma}
 
In the rest of this subsection I will give the key properties and definitions about specializations. 


The following facts are Proposition 2.2.7 and 2.2.15 in 
\cite{Zilber}. The first one is due to more general results of Weglorz \cite{Weglorz}.

\begin{Fact}\label{thm:existence-special}
Let $M$ be a quasi-compact, and complete Zariski structure, and $M^\prime\succeq M$. Then 
there is a total specialisation $\pi:M^\prime \to M$. Moreover, any (partial) specialisation 
can be extended to a total specialisation. 
\end{Fact}

\begin{Fact}\label{thm:existence-univ.special}
Let $M$ be a Zariski structure, and $M^\prime\succeq M$. Then there is a $\kappa$-universal 
specialisation $(M^\prime, \pi)$. Moreover, if $M$ is quasi-compact, then $\pi$ is a total 
specialisation.  
\end{Fact}

The following is one of the natural examples of a specialization.

\begin{Example}\label{ex:acvf}
 Let $K$ be an algebraically closed valued field, and $k$ the residue field. Suppose $k\succeq K$ 
 as fields.Consider $K$ with the language $\mathcal L_{\mbox{\tiny Zar}}(k)$ for algebraic 
 varieties defined over $k$. Then the residue map $res:K\to k$ where $k$ is a specialisation. 
\end{Example}

However, note that $res$ is a partial specialisation in this example. One can easily give an 
example of a total specialisation. 

\begin{Example}[Example 2.2.4 in \cite{Zilber}.]\label{ex:spcl-projective-line}
Let ${}^*\mathbb R$ be the hyperreal numbers. Then the standard part map from ${}^*\mathbb R$ 
to $\mathbb R$ (which is a partial map) induces a total specialisation $\pi: \mathbb P^1{}^*\mathbb R  \to \mathbb P^1\mathbb R$. 
Note that ${}^*\mathbb R$ interprets an elementary extension ${}^* \mathbb C$ of $\mathbb C$, 
therefore one can consider $\pi: \mathbb P^1{}^*\mathbb C\to \mathbb P^1\mathbb C$ in the same 
way, and show that it is a total specialisation.  
\end{Example}

Specializations can be thought as replacements to the topology given by the positive quantifier 
free formulas. This remark can be made precise in the following:

\begin{Definition}
 Let $\pi:M\to M_0$ be a specialization, a definable relation $S\subseteq M_0^n$ is said to be
 $\pi$-closed whenever $\pi({}^*S)\subseteq S$ where ${}^*S$ is the interpretation of $S$ in $M$. 
\end{Definition}

\begin{Fact}[Exercise 2.2.9 in \cite{Zilber}]
The family of $\pi$-closed sets satisfies the topological axioms (1)-(8). 
\end{Fact}

\begin{Fact}[Proposition 2.2.24 in \cite{Zilber}]\label{fct:specialization-topology}
 If $\pi:M\to M_0$ is a universal specialization then any definable relation $S\subset M_0$ is 
 closed if and only if it is $\pi$-closed. 
\end{Fact}

\section{Specialisations and Algebraically Closed Valued Fields}\label{sct:acvf}
Consider the structure $(K,k,res)$ where $K$ is an algebraically closed 
valued field, $k$ its residue field and $res:K\to k$ is the residue map. 
Furthermore, consider $k$ as a Zariski structure with respect to the natural 
Zariski language associated to it. Denote this language as $\mathcal L_{\mbox{\tiny Zar}}$. 
It is clear that $K$ is also an $\mathcal L_{\mbox{\tiny Zar}}$-structure. 
Assume $\Char K = \Char k$. Then both $K$ and $k$ are algebraically 
closed fields over the same prime subfield. Hence one can (uncanonically) 
embed $k$ into $K$, also one can choose this embedding in such a way 
that $res$ is identity on $k$ (by replacing the isomorphic copy of $k$ in 
$K$ with $k$ if necessary). Then by quantifier elimination (Fact~\ref{thm:zar-qe}) 
this embedding is elementary, and we have $k\prec K$.

as a specialization. For this consideration 
we are forced to make some restrictions. As we need to have $k\prec 
K$, we restrict ourselves to the case $\Char K = \Char k$. Given that 
we are in the equi-characteristic, both $K$ and $k$ are algebraically 
closed fields over the same prime subfield. Then one can (uncanonically) 
embed $k$ into $K$, also one can choose this embedding in such a way 
that $res$ is identity on $k$ (by replacing the isomorphic copy of $k$ in 
$K$ with $k$ if necessary). But then by quantifier elimination (Fact~\ref{thm:zar-qe}) 
this embedding is elementary, and we have $k\prec K$.

We will work with the relational language $\mathcal L_{\mbox{\tiny Zar}}$, 
which only has predicates for varieties in $k^n$ (all $n\in\mathbb N$), 
whose relations are Zariski closed subsets of $k^n$ (all $n\in\mathbb N$), 
and with the expanded language $\mathcal L^{res} : =\mathcal L_{\mbox{\tiny Zar}}
\cup\{res\}$ where $res$ will be interpreted as the residue map. Note 
that $k$ is $\mathcal L^{res}$ definable in K by the formula $\exists x (res(x) = y)$. 
Hence we will write our structures as $(K,res)$. From now on $K$ will 
always be an algebraically closed valued field, $k$ the (definable) residue 
field, $res:K\to k$ the residue map, $\mathcal O$ the corresponding valuation 
ring, and $\mathcal M$ the valuation ideal. 

Let $\mathcal L_{\mbox{\tiny div}}=\{+,-,\cdot,0,1,|\}$ be 
the ring language augmented with a binary relation $|$ which is interpreted 
as $x|y$ if and only if $v(x)\leq v(y)$. 
Recall that ACVF is the complete theory given by the following sentences:
\begin{enumerate}
\item $K$ is algebraically closed.
\item Valuation axioms. 
\item There are $x,y\in K^\times$ such that $v(x) < v(y)$, where $v$ is the valuation on $K$.
\item Characteristic of $k$ and $K$.
\end{enumerate}
in the language $\mathcal L_{\mbox{\tiny div}}$ and also in various other 
languages due to Robinson's analysis of algebraically closed valued fields 
\cite{Macpherson, Robinson}.

It is easy to see that one can write sentences 1, 2 and 4 in the language 
$\mathcal L^{res}$, one can also write 3 but it takes a bit more effort. 
First observe that $\mathcal O$ is $\mathcal L^{res}$ definable, by the 
formula $\exists y (res(x) = y)$. Then define $Res: K\times K\to k$ as follows:
\[
Res(x,y) = \left\{ \begin{array}{ll}
 res(xy^{-1}) &\mbox{ if $v(x)\geq v(y)$} \\
  0 &\mbox{ otherwise}
       \end{array} \right.
\]
So $Res$ is $\mathcal L^{res}$ definable by the formula 
\[\psi(x,y,z): (xy^{-1}\in\mathcal O \wedge z=res(xy^{-1})) \vee (xy^{-1}
\not\in\mathcal O \wedge z=0)\] 
Now sentence 3 can can be written as $\exists x \exists y (Res(x,y)=0)$. 
Therefore, one can indeed consider the theory ACVF to be living inside 
the theory specialisations (with some restrictions).

In the next theorem I will show any specialization can be extended in such 
a way that its domain is a valuation ring. This is basically due to Chevalley's 
Place Extension Theorem. However, the essential arguments are already 
present in Weil's Foundations of Algebraic Geometry \cite{Weil}. 

\begin{Theorem}\label{thm:ext-res}
Let $k$ be an algebraically closed field and $K\succeq k$.
Given a specialization $\pi: K\to k$, it defines a residue map, $res: 
K\to k$ of a valuation ring $\mathcal O$.
\end{Theorem}
\begin{proof}
Let $R=\Dom \pi$. Clearly we can assume $R$ is ring containing $k$ 
(if not, take $R$ to be the ring generated by $\Dom{\pi}$. Then $\pi$ 
naturally extends to a ring homomorphism from $R$ to $k$ which is a 
specialization extending $\pi$ to $R$). Then $\pi$ is a surjective ring 
homomorphism with $\Ker\pi =\mathcal V_0$. Hence $\mathcal V_0
\ideal R$ is maximal, and in particular prime. Then by Chevalley's Place 
Extension Theorem, there is a valuation ring $\mathcal O$ to which $\pi$ 
extends. 

Let us denote the extension of $\pi$ to $\mathcal O$ by $res$ and write 
$res:\mathcal O\to k$. Then $\Ker {res}$ is the unique maximal ideal (i.e. 
the valuation ideal) of $\mathcal O$ and $\mathcal O/\Ker{res}\simeq k$. 
\end{proof}

Hence the specialization denoted by $res$ with the domain $\mathcal 
O$ above is a residue map coming from a valuation of the field $K$ 
with residue field $k$.

\begin{Remark}\label{rmk:spcl-residue}
If $\pi:K\to k$ is a specialization of algebraically closed fields, by the 
above theorem we can consider its extension to a residue map $res: 
K\to k$. Therefore whenever we have a specialization $\pi$ we will 
consider the residue map $res$ it defines instead of it. 
\end{Remark}

\subsection{Some Facts on Extensions of Valued Fields}
I will recall some well known facts on extensions of valued fields. They 
are widely available in the literature on valued fields, e.g. in~\cite{vdD2}.
In this subsection, and also in the next one (subsection~\ref{ssect:QE-ACVF}) 
we will consider a valued field as tuple $(F,\mathcal O_F)$ where $F$ is an 
algebraically closed field and $\mathcal O_F$ is a valuation subring of $F$. 
If $F$ is algebraically closed we will call $(F,\mathcal O_F)$ an algebraically 
closed valued field.

\begin{Definition}
Let $(F,\mathcal O_F)$ and $(L,\mathcal O_L)$ be two valued fields. A valued field 
embedding $i:(F,\mathcal O_F)\to (L,\mathcal O_L)$ is a field embedding $i:F\to L$ 
such that $i(\mathcal O_F) = \mathcal O_L$. 
\end{Definition}

\begin{Fact}[Corollary 3.16 in~\cite{vdD2}]\label{fct:3.16}
Let $(F,\mathcal O_F)$ be a valued field. Let $F^{\mbox{\tiny alg}}$ be its algebraic 
closure and $\mathcal O_{F^{\mbox{\tiny alg}}}\subseteq F^{\mbox{\tiny alg}}$ be a 
valuation ring which dominates $\mathcal O_F$. Then any valued field embedding $i:(F,
\mathcal O_F)\to (L,\mathcal O_L)$ to an algebraically closed $L$ can be extended to 
a valued field embedding $(F^{\mbox{\tiny alg}},\mathcal O_{F^{\mbox{\tiny alg}}})\to 
(L,\mathcal O_L)$. 
\end{Fact}

\begin{Fact}[Lemma 3.22 in~\cite{vdD2}]\label{fct:3.22}
Let $(F,\mathcal O_F)$ be a valued field, let $x$ be transcendental over $F$ and let 
$L=F(x)$. Then there is a unique valuation ring $\mathcal O_L\subseteq L$ with $
\mathcal O_F\subseteq \mathcal O_L$ such that $x\in\mathcal O_L$ and $res(x)$ 
is transcendental over the residue field $L_0$. Furthermore, $L_0=M_0(res(x))$ and 
$\Gamma_F = \Gamma_L$ where $\Gamma_F$ and $\Gamma_L$ are the corresponding 
value groups. 
\end{Fact}

\begin{Fact}[Lemma 2.23 in~\cite{vdD2}]\label{fct:2.23}
Let $(F,\mathcal O_F)$ be a valued field with valuation $v:F^\times\to\Gamma_F$, let 
$x$ be transcendental over $F$ and let $L=F(x)$. If $h$ is an element of an ordered abelian 
group extending $\Gamma_F$ such that $nh\not\in\Gamma_F$ for all $n\geq 1$, then 
$v$ extends uniquely to a valuation $w:L^\times\to\Gamma_F+h\mathbb Z$. Also $F_0 = 
L_0$. 
\end{Fact}

\begin{Fact}[Lemma 3.30 in~\cite{vdD2}]\label{ftc:3.30}
Let $(F,\mathcal O_F)\subseteq (L,\mathcal O_L)$ be a valued field extension with $M_0
\subseteq L_0$, also denote the valuation of $L$ by $v$. Let $m_1,\ldots,m_n\in M$ and 
$x\in L\setminus M$. Suppose $v(x-m_i)\in vM^\times$ for all $i=1,\ldots,n$. There there 
is an element $m\in M$ such that for all $i=1,\ldots,n$, $v(m-m_i) = v(x-m_i)$.
\end{Fact}

\subsection{Quantifier Elimination for $Th(K)^{res}$}\label{ssect:QE-ACVF}
Let $k$ be an algebraically closed field considered as a Zariski 
structure in its natural Zariski language $\mathcal L_{\mbox{\tiny Zar}}$. 
Let $K\succeq k$, we will consider the structure $(K,k,\pi)$ in the 
language $\mathcal L_{\mbox{\tiny Zar}}^{\pi,\infty}(P) := \mathcal L_
{\mbox{\tiny Zar}}\cup \{P,\pi,\infty\}$ where $P$ is a unary predicate 
interpreted as $k$, $\pi$ is a function symbol which will be interpreted 
as a specialization $\pi:K\to k$, and $\infty$ is a new constant to be 
interpreted as a new element added to $k$. For dealing with substructures 
in a natural way we will also assume that the language $\mathcal L_{\mbox{\tiny Zar}
}^{\pi,\infty}(P)$ contains function symbols $+$ and $\cdot$ to be 
interpreted as addition and multiplication. Since adding $+$ and $\cdot$ 
to the language $\mathcal L_{\tiny Zar}$ to be interpreted as addition 
and multiplication is a definitional expansion, this is harmless.

Let $Th(K)^{res}$ be the theory given by the following:
\begin{enumerate}
\item $Th(K)$, the $\mathcal L_{\mbox{\tiny Zar}}$ theory of $K$.
\item $\infty$ will be a new element added to the residue field:\\
$P(\infty)\wedge \forall a(P(a)\implies (\infty+a=a+\infty = \infty \wedge a\cdot\infty=\infty\cdot a = \infty))$.
\item $res$ is a non-trivial residue map:\\
$\forall x(res(x)\neq\infty \vee res(x^{-1})\neq\infty) \wedge \exists z (\neg P(z))$
\item The axiom scheme saying $res$ is a specialization:\\
$\{\forall x_1,\ldots, x_n (S(x_1,\ldots, x_n)\implies S(res(x_1,\ldots, x_n))): n\in\mathbb N, S\in
\mathcal L_{\mbox{\tiny Zar}}\}$
\item $res$ is identity on the residue field:\\
$\forall x (P(x)\implies res(x)=x)$
\end{enumerate}

Let $(M,M_0,res)$ be a model of $Th(K)^{res}$. Then $res:M\to M_0\cup
\{\infty\}$ is a place. So $res$ corresponds to a unique valuation ring, 
denote it by $\mathcal O_M$. Therefore $(M, \mathcal O_M)$ is a non-trivial 
valued field. For the sake of notation, when there is no risk of ambiguity, 
instead of $res(m)$ (given that $res$ is defined on $m$), I will write $\bar m$. 
I.e. whenever $res$ is defined on $m\in M$ we define $\bar{m} :=res(m)$. This 
convention will only be valid in this section.

We will use the following fact from~\cite{Henson} for proving $Th(K)^{res}$ 
admits quantifier elimination, in the language $\mathcal L_{\mbox{\tiny Zar}}
^{\pi,\infty}(P)$. 

\begin{Fact}[Theorem 5.7 in~\cite{Henson}]
For a theory $T$ in a language $\mathcal L$ with $\abs{\mathcal L} = \kappa$, 
the following are equivalent:
\begin{enumerate}
\item $T$ admits quantifier elimination.
\item For any two models $\mathcal M$ and $\mathcal N$ of $T$ where $\abs{M}
\leq\kappa$ and $\mathcal N$ is $\kappa^+$-saturated, and for any proper 
substructure $S$ of $\mathcal M$ whit an embedding $i:S\to\mathcal N$, 
one can extend $i$ to an embedding $j:S^\prime\to\mathcal N$ for some 
proper substructure $S^\prime$ of $\mathcal M$ strictly containing $S$. 
\end{enumerate}
\end{Fact}

\begin{Remark}
Quantifier elimination for more complicated theories of algebraically closed 
valued fields are present in the literature (e.g.~\cite{Delon2012,HHM,HHM2, 
Macpherson, vdD2, Robinson}).

In particular, the main argument we used in the proof of quantifier elimination 
for $Th(K)^{res}$ (Lemmas~\ref{lem:valued-residue-preservation}, \ref{lem:substr-to-alc} 
and \ref{lem:QE-complete}) are very close to Hrushovski and Kazdan's proof of 
quantifier elimination for a three sorted theory algebraically closed valued fields~\cite[Lemma 6.3]{HK}. 
It is also similar to a quantifier elimination result by Leloup~\cite{Leloup}. 
\end{Remark}



\begin{Lemma}\label{lem:valued-residue-preservation} 
Let $(S,S_0,res)$ and $(L,L_0,res)$ be two valued fields where $L$ is 
an algebraically closed field. Consider the extension $(S^{\mbox{\tiny alg}},
S^{\mbox{\tiny alg}}_0,res)$ of $(S,S_0,res)$ where $S^{\mbox{\tiny alg}}$ 
is the algebraic closure of $S$. Then any $\mathcal L_{\mbox{\tiny Zar}}^
{\pi,\infty}(P)$-embedding $i:(S,S_0,res)\to (L,L_0,res)$ extends to an 
$\mathcal L_{\mbox{\tiny Zar}}^{\pi,\infty}(P)$-embedding $(S^{\mbox{\tiny alg}},
S^{\mbox{\tiny alg}}_0,res)\to (L,L_0,res)$.
\end{Lemma}
\begin{proof}
First we will prove that there is an $\mathcal L_{\mbox{\tiny Zar}}^{\pi,
\infty}(P)$-embedding 
\[j_1:(SS^{\mbox{\tiny alg}}_0, S^{\mbox{\tiny alg}}_0, res)\to (L,L_0,res)\] 
extending $i$.

Observe that $S_0$ is relatively algebraically closed in $S$. So $S$ is a 
purely transcendental extension of $S_0$. Therefore $S$ and $S^{\mbox
{\tiny alg}}_0$ are linearly disjoint over $S_0$. Moreover, $S\otimes_{S_0} 
S^{\mbox{\tiny alg}}_0$ is a field since $S^{\mbox{\tiny alg}}_0$ is an 
algebraic extension of $S_0$. Hence the compositum $SS^{\mbox{\tiny alg}}_0$ 
is isomorphic (as fields) to the tensor product of fields $S\otimes_{S_0} S^{\mbox
{\tiny alg}}_0$.

We know that $i$ embeds $S$ into $L$ and $S_0$ into the algebraically 
closed field $L_0$. Extend $i$ to a field embedding $j_0:S^{\mbox{\tiny alg}}
_0\to L_0$. Then, $j_0(S^{\mbox{\tiny alg}}_0) = i(S_0)^{\mbox{\tiny alg}}$. 
As above, $i(S_0)$ is relatively algebraically closed in $i(S)$. Therefore 
$i(S)$ and $j_0(S^{\mbox{\tiny alg}}_0)$ are linearly disjoint over $i(S_0)$.

Then there is an isomorphism $j_1:SS^{\mbox{\tiny alg}}_0\to i(S)j_0(S^{
\mbox{\tiny alg}}_0)$ such that $j_1(x) = i(x)$ for all $x\in S$ and $j_1(y) = 
j_0(y)$ for all $y\in S^{\mbox{\tiny alg}}_0$. It is clear that $j_1$ preserves 
$P$ as $j_1(S^{\mbox{\tiny alg}}_0) = i(S_0)^{\mbox{\tiny alg}}\subseteq 
L_0 = P(L)$. Moreover, $j_1$ preserves the map $res$. \todo{WHY??}

Clearly $S^{\mbox{\tiny alg}}$ is an algebraic closure of $SS^{\mbox{\tiny alg}}_0$. 
Extend $j_1$ to a field embedding $j:S^{\mbox{\tiny alg}}\to L$. Since $j$ 
extends $j_1$, it maps $S^{\mbox{\tiny alg}}_0$ into $L_0 = P(L)$. Hence 
$j$ preserves $P$. Moreover, $j^{-1}(\mathcal O_L)$ is a valuation ring 
of $S^{\mbox{\tiny alg}}$ lying above $\mathcal O_S$. Since $S^{\mbox{
\tiny alg}}$ is normal over $S$, there is a $\sigma\in \Aut{S^{\mbox{\tiny alg}}
|S}$ such that $j^{-1}(\mathcal O_L) = \sigma(\mathcal O_{S^{\mbox{\tiny alg}}})$.

Since $SS^{\mbox{\tiny alg}}_0$ is a normal extension of $S$, $\sigma$ fixes 
$SS^{\mbox{\tiny alg}}_0$ setwise. Moreover $\sigma$ fixes $S^{\mbox{\tiny alg}}_0$ 
setwise. Indeed, since $S^{\mbox{\tiny alg}}_0$ is algebraically closed in $SS^
{\mbox{\tiny alg}}_0$, any automorphism of $SS^{\mbox{\tiny alg}}_0$ over 
$S$ fixes $S^{\mbox{\tiny alg}}_0$ setwise. Hence $\sigma$ fixes $S^{\mbox{\tiny alg}}_0$ 
setwise. 

Then $j\sigma:(S^{\mbox{\tiny alg}} , \mathcal O_{S^{\mbox{\tiny alg}}})\to 
(L,\mathcal O_L)$ is a valued field embedding which maps $S^{\mbox{\tiny alg}}_0$ 
into $L_0$, and hence preserves $P$. 

Next we will show that $j\sigma$ preserves the map $res$. Let $\beta := res(b)$. 
Then $\beta$ is the only element in $S_0^{\mbox{\tiny alg}}$ with $v(b-\beta) 
>v(b)$. As $j\sigma$ is a valued field embedding, $v(j\sigma(b) - j\sigma(\beta)) 
> v(j\sigma(b))$. On the other hand, $\bar{j\sigma(b)}$ is the only element 
in $j\sigma(S_0^{\mbox{\tiny alg}})$ such that $v(j\sigma(b) - \bar{j\sigma(b)}) 
> v(j\sigma(b))$. Hence $j\sigma(\beta) = \bar{j\sigma(b)}$.

\end{proof}

\begin{Lemma}\label{lem:substr-to-alc}
Let $(M, M_0, res)$ and $(L,L_0,res)$ be two models of $Th(K)^{res}$ 
with $(L,L_0,res)$ is $\abs{M}^+$-saturated. Let $(S,S_0,res)$ be a 
substructure of $(M, M_0, res)$ and $i:(S,S_0,res)\to (L,L_0,res)$ be an 
$\mathcal L_{\mbox{\tiny Zar}}^{\pi,\infty}(P)$-embedding. Then there 
is an extension $j:(F^{\mbox{\tiny alg}}, F^{\mbox{\tiny alg}}_0, res) \to 
(L,L_0,res)$ of $i$ where $F$ is the field generated by $S$.
\end{Lemma}

\begin{Lemma}\label{lem:QE-complete}
The theory $Th(K)^{res}$ admits quantifier elimination in the language $\mathcal 
L_{\mbox{\tiny Zar}}^{\pi, \infty}(P)$. Moreover the theory $Th(K)^{res}$ is complete.  
\end{Lemma}
\begin{proof}
Let $\mathcal M = (M,M_0,res)$ and $\mathcal N = (N,N_0,res)$ be two 
models of $Th(K)^{res}$ such that $(N,N_0,res)$ is $\abs M^+$-saturated. 
Let $(S,S_0,res)\subseteq \mathcal M$ be a proper substructure and suppose 
$i:\mathcal (S,S_0,res)\to \mathcal N$ is an embedding. We will show that 
we can extend $i$ to an embedding $j:(S^\prime,S^\prime_0,res)\to\mathcal N $ 
of some substructure $S^\prime\subseteq M$ properly containing $S$. 

By the above arguments it is enough to  consider the case where $S$ 
is a proper algebraically closed subfield of $M$, and $i: (S,S_0,res)\to 
(N,N_0,res)$ is an embedding.

\begin{description}[leftmargin=0cm]
 \item[Case 1] $S_0 \neq M_0$. Let $x\in M_0\setminus S_0$. Then 
 $res(x)  = x$ is transcendental over $S_0$ as $S_0$ is algebraically 
 closed. Moreover  $x\not\in S$, and hence $x$ is transcendental over 
 $S$. Since $(N,N_0,res)$  is saturated there is a $y\in N_0$ with $res(y) 
 = y$ and $res(y)\not\in i(S_0)$.
 
Therefore $res(y)$ is transcendental over $i(S_0)$ and $y$ is transcendental 
over $i(S)$. The isomorphism $i:S\to i(S)$ extends to a field isomorphism 
$j:S(x)\to i(S)(y)$, where $j(x) = y$. Moreover, by Fact~\ref{fct:3.22}, $j:(S(x),
\mathcal O_{S(x)})\to (i(S)(y),\mathcal O_{i(S)(y)})$ is a valued field isomorphism 
and $j(res(x)) = j(x) = y = res(y) = res(j(x))$.

Also, observe that $j$ maps $S_0(x)$ to $i(S_0)(y)$. Therefore it preserves 
the predicate $P$.

Next, for any $b\in S(x)$, we will show that $j(res(b)) = res(j(b))$. First 
assume that $b\in\mathcal O_{S(x)}$. If $b\in\mathcal M_{S(x)}$, then 
$res(b) = 0$, and also $j(b)\in\mathcal M_{i(S)(y)}$. Therefore $j(res(b))= 
j(0) = 0 = res(j(b))$. Now assume that $b\in\mathcal O_{S(x)}\setminus\mathcal M_{S(x)}$, 
i.e. $v(b)=0$. Then $b=\frac{f(x)}{g(x)}=\frac{f_0+f_1x+\ldots+ f_nx^n}{g_0+g_1x+\ldots+ g_mx^m}$ 
for some $f(x),g(x)\in S[x]\setminus\{0\}$ with $v(f)=v(g)$. Furthermore 
we may assume that $v(f) = v(g) =0$. Indeed, if this is not the case, divide 
all coefficients of $f$ and $g$ by a coefficient of $g$ which has minimal 
valuation. Then, in particular, all coefficients of $f(x)$ and $g(x)$ are in $
\mathcal O_{S(x)}$.

Then  
\begin{eqnarray*}
&j\left(res\left(\frac{f}{g}\right)\right)& = j(resf(x)(g(x))^{-1})) = \\
&=& j(res((f_0+f_1x+\ldots+ f_nx^n)(g_0+g_1x+\ldots+ g_mx^m)^{-1}))\\
&=&j(res(f_0+f_1x+\ldots+ f_nx^n)(res(g_0+g_1x+\ldots+ g_mx^m))^{-1})\\
&=&j((\bar{f_0}+\bar{f_1}x+\ldots+ \bar{f_n}x^n)(\bar{g_0}+\bar{g_1}x+\ldots+ \bar{g_m}x^m)^{-1}))\\
&=&j((\bar{f_0}+\bar{f_1}x+\ldots+ \bar{f_n}x^n))j(\bar{g_0}+\bar{g_1}x+\ldots+ \bar{g_m}x^m)^{-1}))\\
&=&(i(\bar{f_0})+i(\bar{f_1})y+\ldots+ i(\bar{f_n})y^n)(i(\bar{g_0})+i(\bar{g_1})y+\ldots+ i(\bar{g_m})y^m))^{-1}\\
&=&(\bar{i(f_0)+i(f_1)y+\ldots+ i(f_n))y^n)(i(g_0)+i(g_1)y+\ldots+ i(g_m)y^m)^{-1}})\\
&=&res\left(j\left(\frac{f_0+f_1x+\ldots+ f_nx^n}{g_0+g_1x+\ldots+ g_mx^m}\right)\right)
\end{eqnarray*}

Now assume $b\in S(x)\setminus\mathcal O_{S(x)}$. Then $res(b)=
\infty$ and $j(res(b)) = j(\infty) = \infty$. On the other hand $j(b)\in 
i(S)(y)\setminus\mathcal O_{i(S)(y)}$. Therefore $res(j(b))=\infty$. 
Hence $j(res(b)) = j(\infty) 
= \infty = res(j(b))$.
 
 \item[Case 2] $S\cap M_0 = M_0$ and $vS^\times \neq vM^\times$. 
 Let $g\in vM^\times\setminus vS^\times$. Since $(N,N_0,res)$ is 
 $\abs{M}^+$-saturated, $vN^\times$ is $\abs{M}^+$-saturated as 
 an ordered abelian group, moreover $vN^\times$ is non-trivial. Also, 
 $vi(S^\times)$ and $vS^\times$ are divisible groups. Therefore, there 
 is $h\in vN^\times\setminus viS^\times$ such that for all $s\in S$, 
 $g<v(s)$ if and only if $h<v(i(s))$.
 
Hence the isomorphism $vS^\times\to vi(S^\times)$ extends to the 
following  ordered abelian group isomorphism
 
 \begin{eqnarray*}
 vS^\times +g\mathbb Z&\to & vi(S^\times) + h\mathbb Z \\
 v(s) + gn &\mapsto & v(i(s)) + hn
 \end{eqnarray*}
 
Now pick elements $x\in M^\times$ and $y\in L^\times$ with 
valuations $g$ and $h$ respectively (i.e. $v(x)=g$ and $v(y) = 
h$). Clearly, $x\not\in S$ and $y\not\in i(S)$. Therefore, $x$ and 
$y$ are transcendental over $S$ and $i(S)$ respectively. 

Let $S^\prime := S(x)$, $R:= i(S)$ and $R^\prime:= i(S)(y)$. We can 
extend $i$ to a field isomorphism $j:S^\prime\to R^\prime$. Moreover, 
by Fact~\ref{fct:2.23} $j$ is a valued field embedding. Therefore $j$ 
maps $\mathcal O_{S^\prime}$ isomorphically to $\mathcal O_{R^
\prime}$, and $\mathcal M_{S^\prime}$ to $\mathcal M_{R^\prime}$. 
Hence we can extend $j$ naturally to $j:\mathcal O_{S^\prime}/{\mathcal 
M_{S^\prime}} \to\mathcal O_{R^\prime}/{\mathcal M_{R^\prime}}$.

Let $S^\prime_0 := P(S^\prime)= M_0\cap S^\prime = M_0$ and $R^
\prime_0 := P(R^\prime)= N_0\cap S^\prime\subset N_0$. Then $j$ 
maps $P(S^\prime)$ to $P(R^\prime)$ since it extends $i$, and $i$ 
maps $M_0$ into $N_0$. Hence $j$ preserves the predicate $P$.

It follows that $j$ preserves the map $res$ from a similar argument 
to Case 1 above. Therefore $j$ is also an $\mathcal L_{\mbox{\tiny Zar}}
^{\pi}(P)$-embedding. 
 

\item[Case 3] $S\cap M_0 = M_0$ and $vS^\times = vM^\times$. 
We may assume $vS^\times = vi(S^\times)$ by identifying $(S,S\cap 
M_0,res)$ with its image under $i$. Let $x\in M\setminus S$. Let $f(x)
\in S[x]$ be a non-zero element. Then $f(x) = u\displaystyle\prod_i 
(x-s_i)$ for some $u, s_i\in S$. Therefore $v(f) = v(u) +\displaystyle
\sum_i v((x-s_i))$. Hence the valuation $v_{|_{S(x)}}:S(x)\to vS$ is determined 
by $v_{{|_{S}}}:S\to vS$, and $S\to vS^\times$ given by $s\mapsto 
v(x-s)$.
 
By Fact~\ref{ftc:3.30} and saturation there exists an element $y\in L
\setminus i(S)$ such that $v(y-s) = v(x-s)$ for all  $s\in S$. Next, extend 
the isomorphism $i$ to a field isomorphism $j:S(x)\to i(S)(y)$ by $j(x)
=y$. Then $j$ is a valued field embedding. 

As in Case 2, $P(S(x))= M_0\cap S(x) = M_0$ and $P(i(S)(y))= N_0\cap 
i(S)(y)\subset N_0$. Then $j$ maps $P(S(x))$ to $P(i(S)(y))$ since it 
extends $i$, and $i$ maps $M_0$ into $N_0$. Hence $j$ preserves the 
predicate $P$.  

Again, it follows that $j$ preserves the map $res$ from an argument 
similar to the one used in Case 1 above. Therefore $j$ is also an 
$\mathcal L_{\mbox{\tiny Zar}}^{\pi}(P)$-embedding. 
\end{description} 
  
This establishes quantifier elimination for $Th(K)^{res}$. 
For completeness, consider $(k,k,id)$ where $id:k\cup\{\infty\}\to k
\cup\{\infty\}$ is identity. Take an element $t$ which is transcendental 
over $k$ and adjoin it to $k$. This is a prime substructure of the theory 
$Th(K)^{res}$ (i.e. it embeds into any model of $Th(K)^{res}$). Hence 
this theory is complete, as a theory with quantifier elimination and a prime 
structure is complete.
\end{proof}

\begin{Lemma}\label{lem:saturated-universal}
Let $(N,N_0,res)$ be a $\kappa$-saturated model of $Th(K)^{res}$ (in 
the language $\mathcal L_{\mbox{\tiny Zar}}^{\pi,\infty}(P)$), then $res:N\to 
N_0$ is a $\kappa$-universal specialization.
\end{Lemma}
\begin{proof}
Let $N^\prime\succeq N$ as $\mathcal L{\mbox{\tiny Zar}}$ structures, 
and $A\subseteq N^\prime$ with $\abs{A}<\kappa$. Let $res^\prime:A
\cup N\to N_0$ be a specialisation extending $res$. We may assume that 
$res^\prime(a)\in A$ for each $a\in A$ (if not, we can add those elements 
from $N_0$ to $A$, the resulting set will still have cardinality less than 
$\kappa$).

We need to show that there is an elementary embedding $\sigma: A\to N$ 
over $A\cap N$ such that $res^\prime_{|_{A}} =res\circ\sigma$.

Enumerate $A\setminus N$ as $\{b_i: i<\lambda\}$. Let $A_0=A\cap N$, 
and $\sigma_0:=id= A\cap N\to N$, and let $A_j=A_0\cup\{b_i:j<i\}$. By 
induction we will build a chain of partial elementary embeddings $\sigma_i:
A_i\to L$ such that $\sigma_i\subseteq\sigma_j$ for $i<j$ and $res^\prime_{|_{A_i}} 
=res\circ\sigma_i$. 

Let $i=j+1$ and $\sigma_j:A_j\to N$ be a partial elementary embedding. 
Consider the $\mathcal L_{\mbox{\tiny Zar}}^{\pi,\infty}(P)$ type 
\[p(x)=\{\varphi(x,\sigma_j(\bar a)): N^\prime\models \varphi(b_{j+1},\bar a), 
\bar a \in A_j\}\]

Since $\abs{A_j}<\kappa$, by $\kappa$-saturation of $(N,N_0,res)$, the 
type  $p(x)$ can be realised. is realized by some $c_i\in N$. Then define 
$\sigma_i(b_i) =c_i$. To see that $res^\prime(b_i) = res(\sigma_i(b_i))$, 
observe that $res^\prime(b_i)\in A_0$. Then $res^\prime(b_i)=\beta$ 
for some $\beta\in A_0$. Then the formula $res(x)=\beta$ is in $p(x)$. 
Hence $res(\sigma_i(b_i)) = res(c_i) = \beta_i$. Since $\sigma_j$ is a 
partial elementary embedding, it follows that $\sigma_i$ is partial elementary.  

If $i$ is a limit ordinal, let $\sigma_i:=\displaystyle\bigcup_{j<i}\sigma_j$. 
Then define $\sigma:=\displaystyle\bigcup_{j<\lambda}\sigma_j:A\to N$. 
Clearly, $\sigma$ is an $\mathcal L{\mbox{\tiny Zar}}$ partial elementary 
map, and moreover $res^\prime|_{A}=res\circ\sigma$.
\end{proof}

The following theorem follows from the completeness of $Th(K)^{res}$ and 
Lemma~\ref{lem:saturated-universal}.

\begin{Theorem}\label{thm:res-el-equiv}
Let $(M,M_0,res)$ be a model of $Th(K)^{res}$. Then $(M,M_0,res)$ is 
elementarily equivalent to a model $(N,N_0,res)$ such that $res:N\to N_0$ 
is $\kappa$-universal. 
\end{Theorem}
\begin{proof}
Let $(N,N_0,res)$ be a $\kappa$-saturated model of $Th(K)^{res}$, then 
by Lemma~\ref{lem:saturated-universal} the specialization $res:N\to N_0$ 
is $\kappa$-universal. Since the theory $Th(K)^{res}$ is complete, $(M,M_0,
res)\equiv (N,N_0,res)$.
\end{proof}

\begin{Remark}
The symbol $\infty$ in the language $\mathcal L_{\mbox{\tiny Zar}}^{\pi,
\infty}(P)$ is for convenience. Especially to consider $res$ as a map defined 
on the whole structure; in terms of valuation theory, to consider it as a place 
in the proper way.

However, we can consider the language $\mathcal L_{\mbox{\tiny Zar}}^
{\pi}(P)$ without the $\infty$ symbol, and consider $res$ as a partial map. 
The theory $Th(K)^{res}$ with the exception of axiom 2 can be stated in 
$\mathcal L_{\mbox{\tiny Zar}}^{\pi}(P)$. This axiom only specifies how 
$\infty$ will be treated. Moreover, all the arguments in the proofs of 
Lemmas~\ref{lem:QE-complete} and~\ref{lem:saturated-universal} 
are valid for $\mathcal L_{\mbox{\tiny Zar}}^{\pi}(P)$ (with keeping in mind 
that $res$ is a partial map in this setting, and being cautious accordingly). 
Theorem~\ref{thm:res-el-equiv} will follow in the same way.
\end{Remark}

In Theorem~\ref{thm:ext-res} I have shown that any specialization of an algebraically closed field can be 
extended to have a valuation ring as its domain. In other words, it can be extended to a place. It is easy 
to see that valuation rings are maximal domains for specializations of algebraically closed fields. In other 
words, for algebraically closed fields maximally defined specializations are places. On the other hand places 
are naturally identified with the equivalence classes of valuations. Therefore maximally defined specializations 
$K\to k$ are in one to one correspondence with the equivalence classes of valuations $v:K\to\Gamma$ with the 
residue field $k$.

In particular if we consider a specialization $\pi:\mathbb P^1(K)\to\mathbb P^1(k)$, then $\pi$ can be extended to 
total specialization. Which will again be a place $res:\mathbb P^1(K)\to\mathbb P^1(k)$ such that $res^{-1}(k)$ is a 
valuation ring $R$ and $res^{-1}(\infty) = \mathbb P^1(K)\setminus R$. In this case all total specializations 
$\mathbb P^1(K)\to\mathbb P^1(k)$ are in a one to one correspondence with the equivalence classes of valuations 
$v:K\to\Gamma$ with residue field $k$ as above. A longer, but a more geometric and constructive proof of this fact 
is given by de Piro \cite{de Piro}.

\begin{Remark}
 For complete Zariski structures any specialization can be extended to a total specialization. Hence the notion 
 of maximally defined specialization coincides with total specialization in this case. 
\end{Remark}

\section{Specialisations over Algebraic Varities}\label{sct:varieties}

In this section we will consider the universality of specializations of 
algebraic (quasi-projective) varieties.

In this section $G$ will denote an algebraic (quasi-projective) variety 
defined over a field $k$ which is a subfield of an algebraically closed 
field $K$. $G(K)$ -$K$ points of $G$- can be thought as a Zariski structure 
in a natural way in the language $\mathcal L_{\mbox{\tiny Zar}}$, which 
only has predicates for $k$-definable subvarieties of $G(K)^n$ (all 
$n\in\mathbb N$).

It is a well known fact that $G(K)$ is bi-interpretable with $K$ where 
$K$ has the structure of an algebraically closed field together with 
named elements from $k$. This bi-interpretation induces an isomorphism 
between $Aut(K)$ and $Aut(G)$ \cite[Chap. 3, Prop. 1.4. part (ii)]{Pillay}\footnote{
Note that in the proposition referenced here it is actually proved that 
two countable, $\omega$-categorical structures are bi-interpretable 
if and only if they have homeomorphic automorphism groups. However 
the direction we required above is true without assuming countability 
and $\omega$-categoricity, as it can be seen from the proof.}, hence 
we may think that $Aut(K)$ acts on $G$ in the same way $Aut(G)$ does.

Let $A\subseteq G(K)$. Since $G(K)$ is $k$-definably covered by (quasi) 
affine varieties we can consider the subfield generated by $A$ and $k$, 
which, we will denote by $k(A)$. To be more precise $k(A) = k(\bar a : 
[\bar a]\in A\cap U_i )$ where $[\bar a]$ is the homogeneous coordinates 
of a point and $U_i = \{[\bar a] : a_i = 1\}$.

\begin{Lemma}\label{lem:acl-varities}
Let $A\subseteq G(K)$, then $acl_G(A) = G(k(A)^{\mbox{\tiny alg}})$.
\end{Lemma}
\begin{proof}
$g\in acl_G(A)\iff$ $g$ has finitely many conjugates under $Aut(G/A)\iff g$ has finitely 
many conjugates under $Aut(K/A) \iff g\in G(k(A)^{\mbox{\tiny alg}})$. 
\end{proof}

\subsection{Specialisation Induced by the Residue Map}\label{ssect:Res}
Let $res: K\to k$ be a specialization as in Section~\ref{sct:acvf}, and let $G$ 
be an algebraic variety defined over $k$. Then, $res$ gives rise to a specialization 
$Res: G(K)\to G(k)$. In the case $G=\mathbb P^1$ we have already seen this 
specialization in Example~\ref{ex:spcl-projective-line}. Here I will recall this 
example in more detail and generalize this construction to quasi-projective varieties. 
In particular, the argument below will explicitly show that the induced specialisation 
on $\mathbb P^n$ is total. 

Recall that $\mathbb P^1 = U_0\cup H_0$ where $U_0=\{(x_0:x_1): x_0 =1\}$ 
and $H_0=\{(x_0:x_1): x_0 = 0\}$. Then $H_0$ consists of a single point, 
denoted by $\infty$, and $U_0$ is isomorphic to $\mathbb A^1$ via $(1:x_1)
\mapsto x_1$. 

Therefore we can write $\mathbb P^1(K)=\mathbb A^1(K)\cup\{\infty\} = K\cup
\{\infty\}$ and similarly $\mathbb P^1(k) = k\cup\{\infty\}$. Then one can define 
$Res:\mathbb P^1(K)\to \mathbb P^1(k)$ for $(1:x_1)\in U_0(K)$ as $Res(1:x_1)
= (1:res(x_1))$ given $x_1\in\Dom{res}$. One can extend $Res$ to the whole of 
$\mathbb P^1(K)$: Define $Res(\infty) = \infty$, and for all $(1:x_1)\in U_1$ with 
$x_1\not\in\Dom{res}$, define $Res(1:x_1) = \infty$.


Next we will consider $\mathbb P^n$. Recall that $\mathbb P^n = U_0\cup
\ldots\cup U_n$ where $U_i=\{(x_0:\ldots:x_n): x_i =1\}$ for $i=0,\ldots, n$. Here, 
$U_i$ is isomorphic to $\mathbb A^n$ via $(x_0:\ldots:x_n)\mapsto (x_0,\ldots,x_
{i-1},x_{i+1}, \ldots, x_n)$.

Therefore define $Res:\mathbb P^n(K)\to \mathbb P^n(k)$ as follows: For $(x_0:\ldots
: x_n)\in U_i(K)$ define $Res(x_0:\ldots: x_n) = (res(x_0):\ldots:res(x_{i-1}):1:res(x_{i+1}):
\ldots: res(x_n))$.

\begin{Remark}
For any $(x_0:\ldots:x_n)\in\mathbb P^n(K)$ there is a $\lambda\in K$ such that $\lambda 
x_0,\ldots, \lambda x_n$ are in the valuation ring $\Dom{res}$ but not all of them are in the 
valuation ideal, i.e. there is some $j$ such that $res(\lambda x_j) \neq 0$. Hence $(\lambda x
_j)^{-1}\in\Dom{res}$. Then $(x_0:\ldots:x_n) = (\lambda x_0:\ldots:\lambda x_n) = (\frac{
\lambda x_0}{\lambda x_j}:\ldots:\frac{\lambda x_{j-1}}{\lambda x_j}:1:\frac{\lambda x_{j-1}}
{\lambda x_j}:\ldots:\frac{\lambda x_n}{\lambda x_j}) =(\frac{x_0}{ x_j}:\ldots:\frac{ x_{j-1}}
{ x_j}:1:\frac{ x_{j-1}}{ x_j}:\ldots:\frac{x_n}{x_j})$ and $\frac{x_0}{ x_j},\ldots ,\frac{ x_{j-1}}
{ x_j},\frac{ x_{j-1}}{ x_j},\ldots,\frac{x_n}{x_j}\in\Dom{res}$. 

Therefore, any point $x\in\mathbb P^n(K)$ can be written in homogeneous coordinates as 
$x = (x_0:\ldots: x_{j-1}:1:x_{j-1}:\ldots:x_n)\in U_j$ for some $j$ with $x_0,\ldots,x_{j-1},1,x_
{j-1},\ldots,x_n\in\Dom{res}$. It immediately follows that $Res$ is defined on every point of 
$\mathbb P^n(k)$, i.e. $\Dom{Res} = \mathbb P^n(K)$.
\end{Remark}


Next, as we define $Res$ separately on each of the sets $U_i$, we need to check that it is well 
defined on the intersection of such sets. This will be immediate from the following remark: 

%
%

\begin{Remark} 
One can alternatively define $Res:\mathbb P^n(K)\to \mathbb P^n(k)$ as follows.
Let $(x_0:\ldots, x_n)\in \mathbb P^n(K)$. There is a $\lambda\in K$ such that $\lambda x_0, 
\ldots, \lambda x_n\in\Dom{res}$ such that $res(\lambda x_i)\neq 0$ for some $i=1,\ldots,n$. 
Then define $Res(x_0:\ldots : x_n) = (res(\lambda x_0):\ldots:res(\lambda x_n))$. 

The definition of $Res$ as given above does not depend on $\lambda$. For if $\mu\in K$ is an 
other element such that $\mu x_0,\ldots,\mu x_n\in\Dom{res}$ with $res(\mu x_j) \neq 0$, 
it follows that $(res(\lambda x_0):\ldots:res(\lambda x_n)) = (res(\mu x_0):\ldots:res(\mu x_n))$.  
\end{Remark}

Clearly, $Res$ defined in the remark above coincides with the map $Res$ defined separately on 
each $U_i$. Moreover, the map $Res$, is well defined on the intersections $U_i\cap U_j$ 
for all $i,j$ by the same remark.

Next, we need to check that $Res:\mathbb P^n(K)\to\mathbb P^n(k)$ is a specialization. 
Since $res:k\to k$ is identity, $Res:\mathbb P^n(k)\to\mathbb P^n(k)$ is also identity. 
Let, $S\subseteq (\mathbb P^n(K))^m$ be a closed subset and let $\bar x\in\Dom{Res}
\cap S$. Write $\bar x = (z_1,\ldots, z_m)$ where $z_i = (z_{i0}:\ldots:z_{in})$ for all $i=1
\ldots, m$. There is a finite set of homogeneous polynomials 
\[\{f_j(Z_{01},\ldots,Z_{0n},
\ldots, Z_{m0},\ldots, Z_{mn})\}\subseteq k[Z_{01},\ldots,Z_{0n},\ldots, Z_{m0}, \ldots, 
Z_{mn}]\] 
such that $f_i(z_{01},\ldots,z_{0n},\ldots, z_{m0},\ldots, z_{mn}) = 0$ for all 
$i=1,\ldots,m$.
  
Since $res:K\to k$ is a specialization, $f_i(res(\bar z)) = 0$ for all $i$, where $\bar z = 
(z_{01},\ldots,z_{0n},\ldots, z_{m0}, \ldots, z_{mn})$. Therefore, 
$Res(\bar x)\in S$.
For a quasi-projective variety $G\subseteq \mathbb P^n(k)$, the restriction of $Res:
\mathbb{P}^n(K)\to\mathbb{P}^n(k)$ gives a specialisation.

\subsection{Specialisation Induced on the Field from a Variety}
 
In this subsection I will show that a maximal specialization on $G$ comes from 
a maximal specialization $\pi_K:K\to k$. The first main statement of this subsection 
is the following:

{
\renewcommand{\theTheorem}{\ref{prop:characterization-spcl-affine-varieties}}
\begin{Proposition}
Let $G$ be an affine variety defined over $k$. Let $\pi_G:G(K)\to G(k)$ be a 
maximal specialization. Then there is a maximal specialization $\pi_K:K\to k$ 
such that $\Dom{\pi_G} = G(K)\cap (\Dom{\pi_K})^n$ and for any $\bar x\in 
\Dom{\pi_G}$, its image $\pi_G(\bar x) = (x_1^{\pi_K},\ldots,x_n^{\pi_K})$.
\end{Proposition}
\addtocounter{Theorem}{-1}
}

To prove Proposition~\ref{prop:characterization-spcl-affine-varieties}, one first 
needs some preparation. Let $k$ be and algebraically closed field (seen as a 
Zariski structure) and $K\succeq k$. Let $G$ an affine variety over $k$ and 
$\pi_G:G(K)\to G(k)$ be a maximal specialization.We can interpret $k$ in $G(k)$ 
(and similarly $K$ in $G(K)$) via the projections  $\pr_i:G(k)\to k$ for $i=1,
\ldots,n$. Clearly $\pr_i(G(k))=k\setminus X_i$ for some finite $X_i\subseteq k$. 
Moreover, the equivalence relation $\sim_i$ on $G$ for each $i$, defined 
by $(x_1,x_2,\ldots, x_n)\sim_i (x_1^\prime,x_2^\prime,\ldots,x_n^\prime)$ if 
and only if $\pr_i(x_1,x_2,\ldots, x_n) = \pr_i(x_1^\prime,x_2^\prime,\ldots,
x_n^\prime)$ is closed. 

For each $i=1,\ldots,n$, we define  $\pi_i:K\setminus X_i\to k\setminus X_i$ 
as follows: first, define $\Dom{\pi_i}:=\pr_i(\Dom{\pi_G})$. Second, For any 
$x\in\pr_i(\Dom{\pi_G})$ define $\pi_i(x)$ as follows: Pick an element $(x_1,
x_2,\ldots,x_n)\in\pr_i^{-1}(x)\cap \Dom{\pi_G}$. Then define $\pi_i(x)=\pr_i(\pi_G
(x_1,x_2,\ldots,x_n))$. It is clear that $\pi_i$ is well defined for each $i$, since 
$\sim_i$ is a closed equivalence relation. Next we will show that $\pi_i$ is a specialization.

Note that, $K\setminus X_i$ and $k\setminus X_i$ are quasi-affine varieties, 
and we will consider them with the natural Zariski topologies on their Cartesian 
powers as Zariski geometries. 

\[\xymatrix{ G(K) \ar@{->}[r]^{\pi_G}\ar@{->}[d]^{\pr_i}  &G(k)\ar@{->}[d]^{\pr_i}\\
K_i:= K\setminus X_i \ar@{->}[r]^{\pi_i} & k_i := k\setminus X_i}\]

\begin{Proposition}\label{prop:component-spcl}
For each $i\in\{1,\ldots, n\}$, $\pi_i:K_i\to k_i$ is a specialization.
\end{Proposition}
\begin{proof}
Let $S$ be a closed set of $K_i:=K\setminus X_{i0}$ and $\bar a\in S\cap
\Dom{\pi_i}$. We want to show that $\pi_i(\bar a)\in S$. Since $\pr_i$ 
is continuous, $T:= \pr_i^{-1}(S)$ is closed, and since $\bar a\in\Dom{\pi_i}$, 
there is $\bar\alpha\in\Dom{\pi_G}$ such that $\pr_i(\bar\alpha) = \bar a$. 
Then $G(K_i)\models T(\alpha)$, and since $\pi_G$ is a specialization, $G(k_i)
\models T(\pi_G(\bar \alpha))$.

Then, by definition of $\pr_i$, we have $\pr_i(\pi_G(\alpha))\in\pr_i(T) = S$. 
Therefore, by definition, $\pi_i(\bar a)\in S$. Hence $\pi_i$ is a specialization.
\end{proof}

Therefore we can think of $\pi_G:G(K)\to G(k)$ in terms of the specializations 
$\pi_1,\ldots, \pi_n$. Let $\bar x\in\Dom{\pi_G}$ and say $\pi_G(\bar x) = 
\pi_G(x_1,\ldots,x_n) = (y_1,\ldots,y_n)$. By construction $\pi_1(x_1) = 
\pr_1(\pi_G(\bar x)) = y_1, \ldots, \pi_n(x_n) = \pr_1(\pi_G(\bar x)) = y_n$. 
Then we can write $\pi_G(x_1,\ldots,x_n) = (x_1^{\pi_1},\ldots, x_n^{\pi_n})$.

\begin{Proposition}\label{prop:intrsction-domains}
For any $i,j \in \{1,\ldots, n\}$, the specializations $\pi_i$ and $\pi_j$ are the same 
on the intersection $\Dom{\pi_i}\cap\Dom{\pi_j}$ of their domains; i.e $\pi_i = 
\pi_j$ on $\Dom{\pi_i}\cap\Dom{\pi_j}$.
\end{Proposition}
\begin{proof}
First of all observe that since $\Dom{\pi_i}\cap\Dom{\pi_j} \supseteq 
k\setminus X_i\cup X_j$, it is not empty.

Let $z\in\Dom{\pi_i}\cap\Dom{\pi_j}$. Then there are tuples $\bar x = (x_1,
\ldots, x_i,\ldots,x_n)\in\Dom{\pi_G}$ and $\bar y =(y_1,\ldots, y_j,\ldots,
y_n)\in\Dom{\pi_G}$ such that $x_i = y_j = z$. 

Define 
\[C_{ij} = \{(\bar x, \bar y)\in (G(K))^2 : x_i=y_j\}\]
Then $\bar x,\bar y\in C\cap\Dom{\pi_G}$. Clearly $C$ is a closed subset of 
$(G(K))^2$. Then, since $\pi_G$ is a specialization, $\pi_G(\bar x), \pi_G(\bar y)
\in C$. Then $\pi_i(x_i) = \pi_j(y_j)$. Hence $\pi_i(z) = \pi_j(z)$. 
\end{proof}

Define $D := \Dom{\pi_1}\cup\ldots\cup\Dom{\pi_n}$. Next we will define a 
common extension of the specializations $\pi_1,\ldots,\pi_n$ to $D$. Define 
$\pi_K^0:D\to k\setminus X$ where $X=X_1\cap\ldots\cap X_n$ as follows: 
let $x\in D$, then by definition there is an $i = 1,\ldots, n$ such that $x\in
\Dom{\pi_i}$. Define $\pi_K^0(x) = \pi_i(x)$. Since $\pi_1,\ldots, \pi_n$ are 
the same on the intersection $\Dom{\pi_1}\cap\ldots\cap\Dom{\pi_n}$, this 
is well defined. We now need to show that $\pi_K^0$ is a specialization.

\begin{Proposition}\label{prop:common-ext}
$\pi_K^0:D\to k\setminus X$ is a specialization.
\end{Proposition}
\begin{proof}
Let $S$ be a closed set of $K\setminus X$ and let $\bar z = (z_1,\ldots,z_n)
\in D^n$ be such that $\models S(z_1,\ldots,z_n)$. We will show that $\models 
S(z_1^{\pi_K^0},\ldots,z_n^{\pi_K^0})$.

For each $i=1,\ldots,n$ there is a $j_i\in\{1,\ldots,n\}$ such that $z_i\in\Dom
{\pi_{j_i}}$, since $\bar z\in D^n$. Then for all $i$ there is a tuple $\bar x_i = 
(x_{1i},\ldots,x_{ji},\ldots,x_{ni})\in\Dom{\pi_G}$ such that $z_i = x_{ji}$.

Next, look at the set 
\[C=\{(\bar x_1,\ldots,\bar x_n)\in G(K)^n: S(x_{j1},\ldots,x_{jn})\}\]
Then $C$ is a closed subset of $G(K)^n$. Therefore, since $\pi_G$ is a specialization, 
$\models C(\bar x_1^{\pi_G},\ldots,\bar x_n^{\pi_G})$. Recall that $\bar x_i^
{\pi_G} = (x_{1i}^{\pi_1},\ldots, x_{ji}^{\pi_{ji}},\ldots, x_{ni}^{\pi_n})$. Then in 
particular, $S(x_{j1}^{\pi_{j1}},\ldots,x_{jn}^{\pi_{jn}})$. But, $S(x_{j1}^{\pi_{j1}},
\ldots,x_{jn}^{\pi_{jn}}) = S(x_{j1}^{\pi_K^0},\ldots,x_{jn}^{\pi_K^0})$, since $\pi_
K^0$ is a common extension of all $\pi_1,\ldots,\pi_n$. Then $\models S(z_1^{\pi_
K^0},\ldots,z_n^{\pi_K^0})$ as $z_i = x_{ji}$. Hence $\pi_K^0$ is a specialization.
\end{proof}

The specialization $\pi_K^0:D\to k\setminus X$ is not defined on the finite 
subset $X\subset k$. However, as $X\subset k$ and since any specialization 
$K\to k$ must be identity on $k$ we can canonically extend $\pi_K^0$ to 
a specialization $\pi_K^0:K\to k$ simply by defining it to be identity on $X$. 

From now on, whenever we write $\pi_K^0:K\to k$, it will always refer to 
the specialization of $K$ with domain $D\cup X$, and range $k$ (so that 
$\pi_K^0:D\cup X\to k$ is onto). Further, the specialization $\pi_K^0:K
\to k$ induces a specialization $\pi_G^0:G(K)\to G(k)$ in the natural way:
Define $\Dom{\pi_G^0} := G(K)\cap(\Dom{\pi_K^0})^n$, and for all $\bar x = 
(x_1,\ldots,x_n)\in\Dom{\pi_G^0}$, define $\pi_G^0(\bar x) := (x_1^{\pi_K^0},
\ldots, x_n^{\pi_K^0})$.   

\begin{Proposition}\label{prop:spcl-ind-by-extension}
The specialization $\pi_G^0:G(K)\to G(k)$ induced by $\pi_K^0:K\to k$ 
is an extension of the specialization $\pi_G:G(K)\to G(k)$, and $\pi_G^0 
= \pi_G$ on $\Dom{\pi_G}$.
\end{Proposition}
\begin{proof}
Let $\pi_G^0:G(K)\to G(k)$ be the specialization induced by $\pi_K^0$. 
Then by construction, $\Dom{\pi_G^0} = G(K)\cap(\Dom{\pi_K^0})^n$, 
and for all $\bar x = (x_1,\ldots,x_n)\in\Dom{\pi_G^0}$, $\pi_G^0(\bar x) 
= (x_1^{\pi_K^0},\ldots, x_n^{\pi_K^0})$.

Next, we will show that $\Dom{\pi_G}\subseteq\Dom{\pi_G^0}$. Let 
$\bar x = (x_1,\ldots,x_n)\in\Dom{\pi_G}$. Then $x_i\in\Dom{\pi_i}$ for 
all $i = 1,\ldots,n$. Then $x_1,\ldots,x_n\in\Dom{\pi_K^0}$, since $\pi_K^0$ 
is a common extension of all of the specializations $\pi_1,\ldots,\pi_n$. Hence 
$(x_1,\ldots,x_n)\in\Dom{\pi_G^0}$.

Last, we will show that $\pi_G^0 = \pi_G$ on $\Dom{\pi_G}$. Let $\bar x 
= (x_1,\ldots,x_n)\in\Dom{\pi_G}$. Then $x_i\in\Dom{\pi_i}$ for all $i = 
1,\ldots,n$ as before. Hence
\[\pi_G^0(x_1,\ldots,x_n) = 
(x_1^{\pi_K^0},\ldots,x_n^{\pi_K^0}) = (x_1^{\pi_1},\ldots,x_n^{\pi_n}) 
=\pi_G(x_1,\ldots,x_n)\]

Which concludes that $\pi_G^0$ is an extension of $\pi_G$.  
\end{proof}

Note here that if $\pi_G:G(K)\to G(k)$ is a maximal specialization we will have 
$\Dom{\pi_G} = \Dom{\pi_G^0}$ and $\pi_G^0 = \pi_G$. 

%

At this point we are ready to prove Proposition~\ref{prop:characterization-spcl-affine-varieties}. 
However, before proceeding to the proof of Proposition~\ref{prop:characterization-spcl-affine-varieties}, 
we will prove another result. 

\begin{Proposition}\label{prop:affine-univ-1}
Let $\pi_G:G(K)\to G(k)$ be a maximal specialisation. If the specialization 
$\pi_K^0:K\to k$ induced by $\pi_G$ (as defined before Proposition~\ref{prop:common-ext}) 
is $\kappa$-universal, then $\pi_G$ is $\kappa$-universal. 
\end{Proposition}
\begin{proof}
Suppose that $\pi_K^0$ is $\kappa$-universal. We will show that $\pi_G$ is 
also $\kappa$-universal. Let $G(K^\prime)\succeq G(K)$, and $A\subseteq 
G(K^\prime)$ with $\abs{A}<\kappa$. Let $\pi_{G^\prime}:G(K)\cup A\to G(k)$ 
be a specialisation extending $\pi_G$. 

By Propositions~\ref{prop:component-spcl}, \ref{prop:intrsction-domains}, 
\ref{prop:common-ext}, $\pi_{G^\prime}$ induces a specialisation $\pi_{K^\prime}^0:
K\to k$ such that $\pi_{G^\prime}(x_1,\ldots,x_n) = (x_1^{\pi_{K^\prime}^0},
\ldots, x_n^{\pi_{K^\prime}^0})$. 

Next, we claim that $\pi_{K^\prime}^0$ extends $\pi_K^0$. Indeed, let $a\in
\Dom{\pi_K^0}$. Then $a\in\Dom{\pi_i}$ for some $i$. Therefore there is a tuple 
$a_1,\ldots,a,\ldots,a_n\in\Dom{\pi_G}$. 

By construction of $\pi_K^0$, we have 
\[\pi_G(a_1,\ldots,a,\ldots,a_n)=(a_1^{\pi_K^0},\ldots,a^{\pi_K^0},\ldots,
a_n^{\pi_K^0})\] 
Similarly 
\[\pi_{G^\prime}(a_1,\ldots,a,\ldots,a_n)=(a_1^{\pi_{K^\prime}^0},
\ldots,a^{\pi_{K^\prime}^0},\ldots,a_n^{\pi_K^0})\]

But $\pi_{G^\prime}(a_1,\ldots,a,\ldots,a_n) = \pi_G(a_1,\ldots,a,\ldots,a_n)$, 
since $\pi_{G^\prime}$ extends $\pi_G$. Therefore 
\[(a_1^{\pi_K^0},\ldots,a^{\pi_K^0},\ldots,a_n^{\pi_K^0}) = (a_1^{\pi_{K^\prime}^0},
\ldots,a^{\pi_{K^\prime}^0},\ldots,a_n^{\pi_K^0})\]
Hence $\pi_{K^\prime}^0(a) = \pi_K^0(a)$ as desired.

Since $\pi_K^0$ is $\kappa$-universal, there is an embedding $\sigma_K:B
\to K$ such that $\pi_K^0=\pi_{K^\prime}^0\circ \sigma_K$ on $B$ where 
$B=\displaystyle\bigcup_{i=1}^n\pr_i(A)$. Define $\sigma:A\to G(K)$ as 
$\sigma(x_1,\ldots,x_n) = (\sigma_K(x_1),\ldots,\sigma_K(x_n))$.

Next, we will show that $\pi_{G^\prime}=\pi_G\circ\sigma$ on $A$. Let $(a_1,
\dots,a_n)\in A$. Then 
\begin{eqnarray*}
\pi_{G^\prime}(a_1,\dots,a_n) &=& (a_1^{\pi_{K^\prime}^0},
\dots,a_n^{\pi_{K^\prime}^0})= ((\sigma_K(a_1))^{\pi_K^0},\dots,
(\sigma_K(a_n))^{\pi_K^0})\\ 
&=& \pi_G(\sigma_K(a_1),\ldots,\sigma_K(a_n)) = \pi_G\circ\sigma
(a_1,\ldots,a_n)
\end{eqnarray*}
\end{proof}

\begin{Remark}
Let $G$ be an affine variety defined over $k$. Let $res:K\to k$ be $\kappa$-universal 
specialisation. Consider the specialisation $Res:G(K)\to G(k)$ induced by 
$res$ is $\kappa$-universal: $\Dom{Res} = (\Dom{res})^n\cap G(k)$, and 
for any $\bar a\in\Dom{Res}$, we define $Res(\bar a) = (res(a_1),\ldots,res(a_n))$. 
Clearly, the specialisation $Res_K^0:K\to k$ induced by $Res$ is a restriction 
of $res$. If, $Res_K^0 = res$, then $Res$ is $\kappa$-universal.  
\end{Remark}

Now we proceed to the proof of the first main result of this section.

\begin{Proposition}\label{prop:characterization-spcl-affine-varieties}
Let $G$ be an affine variety defined over $k$. Let $\pi_G:G(K)\to G(k)$ be a 
maximal specialization. Then there is a maximal specialization $\pi_K:K\to k$ 
such that $\Dom{\pi_G} = G(K)\cap (\Dom{\pi_K})^n$ and for any $\bar x\in 
\Dom{\pi_G}$, its image $\pi_G(\bar x) = (x_1^{\pi_K},\ldots,x_n^{\pi_K})$.
\end{Proposition}
\begin{proof}[Proof of Proposition~\ref{prop:characterization-spcl-affine-varieties}]
Let $\pi_G:G(K)\to G(k)$ be a maximal specialization. Look at the coordinate 
projections $\pr_i:G(K)\to K\setminus X_i$ as before and define $\pi_i:K
\setminus X_i\to k\setminus X_i$ as follows: Domain of $\pi_i$ is $\pr_i
(\Dom{\pi_G})$ and for any $x\in\pr_i(\Dom{\pi_G})$ there is an element 
$(x_1,\ldots,x_n)\in\pr_i^{-1}(x)\cap\Dom{\pi_G}$. Define $\pi_i(x) = \pr_i
(\pi_G(x_1,\ldots,x_n))$. By Proposition~\ref{prop:component-spcl}, $\pi_i$ is 
a specialization for all $i$.

Let $D := \Dom{\pi_1}\cup\ldots\cup\Dom{\pi_n}$. We define $\pi_K^0:D\to k
\setminus X$ as follows: for any $x\in D$, there is an $i = 1,\ldots, n$ such that 
$x\in\Dom{\pi_i}$. Then define $\pi_K^0(x) = \pi_i(x)$. By Proposition~
\ref{prop:intrsction-domains}, $\pi_K^0$ is well defined and by Proposition~\ref
{prop:common-ext} it is a specialization.

Extend $\pi_K^0$ to $X$ by defining $\pi_K^0(x) = x$ for all $x\in X$. We can 
do this since $X\subset k$. Then we get an extension of $\pi_K^0$ to a 
specialization $\pi_K^0:K\to k$. 

If $\pi_K^0$ is a maximal specialization, take $\pi_K$ to be $\pi_K^0$. Then, 
by Proposition~\ref{prop:spcl-ind-by-extension} and the observation following 
it, we are done. 

So, suppose $\pi_K^0$ is not maximal. First, by extending $\pi_K^0$ to the ring 
generated by $\Dom{\pi_K^0}$ we may assume that $\Dom{\pi_K^0}$ is a ring. 
In this case $\pi_K^0$ is a ring homomorphism. Then by Chevalley's Extension 
Theorem, we may extend $\pi_K^0$ to a residue map, call it $\pi_K$. It is clear 
that $\pi_K$ is a maximal specialization on $K$.

Next, let $\pi_G^+:G(K)\to G(k)$ be the specialization induced on $G(K)$ by 
$\pi_K$. I.e. $\Dom{\pi_G^+} = G(K)\cap(\Dom{\pi_K})^n$, and if $\bar x = 
(x_1,\ldots, x_n)\in\Dom{\pi_G^+}$, then $\pi_G^+(\bar x) := (x_1^{\pi_K}, 
\ldots, x_n^{\pi_K})$.

Then again by Proposition~\ref{prop:spcl-ind-by-extension} and the observation 
following it, $\Dom{\pi_G} = \Dom{\pi_G^+} = G(K)\cap (\Dom{\pi_K})^n$ and 
$\pi_G = \pi_G^+$. Therefore $\pi_G(\bar x) = \pi_G^+(\bar x) = (x_1^{\pi_K},
\ldots,x_n^{\pi_K})$ for any $\bar x\in\Dom{\pi_G}$. 
\end{proof}

%
%

Next, we will consider the case where $G$ is a quasi-projective variety 
defined over $k$. Therefore $G$ is an open subset of some closed projective 
set $X\subseteq\mathbb P^n$. Let $\mathbb A_i^n:=\{(z_0,\ldots,z_n): 
z_i\neq 0\}$. We know that $\mathbb P^n =\displaystyle\bigcup_{i=0}^n
\mathbb A_i^n$. Also, $\mathbb A_i^n$ is isomorphic to $\mathbb 
A^n$ for all $i$ via 
\begin{eqnarray*}
h_i :\mathbb A_i^n &\to & \mathbb A^n\\
(z_0,\ldots,z_n)&\mapsto & \left(\frac{z_0}{z_i},\ldots,\frac{z_n}{z_i}\right)
\end{eqnarray*}
where the tuple $\left(\frac{z_0}{z_i},\ldots,\frac{z_n}{z_i}\right)$ does not 
contain the $i^{\mbox{\tiny th}}$ coordinate. 

It follows that $G=\displaystyle\bigcup_{i=0}^n U_i$ where $
U_i=\mathbb A_i^n\cap G$. Clearly each $\mathbb A_i^n$ is an open subset of 
$\mathbb P^n$. Therefore each $U_i$ is an open subset of $G$. Moreover 
one can see that each $U_i$ is a locally closed subset of $\mathbb A_i^n$. 
Since $\mathbb A_i^n$ is isomorphic to $\mathbb A^n$, each $U_i$, via 
restriction of $h_i$, is isomorphic to a locally closed subset of $\mathbb A_n$, 
call it $V_i$. So each $U_i$ is a quasi-affine variety. The sets $U_i$ are called 
\emph{affine pieces of $G$}, and a particular $U_i$ will be called the 
$i^{\mbox{\tiny th}}$ affine piece.  

\begin{Proposition}\label{prop:characterization-spcl-qp-varieties}
Let $G$ be a quasi-projective variety, and $\pi_G:G(K)\to G(k)$ be a specialization. 
Then $\pi_G$ induces a specialization $\pi_{U_i}:U_i(K)\to U_i(k)$ for each 
$i$ where $U_i$ is the $i^{\mbox{\tiny th}}$ affine piece of $G$. 
\end{Proposition}
\begin{proof}

Define $\pi_{U_i}:U_i(K)\to U_i(k)$ as follows: $\pi_{U_i}$ is not defined on 
$U_i \setminus \Dom{\pi_G}$. Let $x\in U_i(K)\cap\Dom{\pi_G}$. If $\pi_G(x)
\in U_i(k)$, define $\pi_{U_i}(x) := \pi_G(x)$. If $\pi_G(x)\not\in U_i(k)$, then 
$\pi_{U_i}$ will not be defined on $x$. Hence $\Dom{\pi_{U_i}} = \{x\in\Dom{\pi_G}
: \pi_G(x)\in U_i(k)\}$.

Next, we will show that $\pi_{U_i}:U_i(K)\to U_i(k)$ is a specialization. Since 
$\pi_G$ is identity on $G(k)$, it is clear that $\pi_{U_i}$ is identity on $U(k)$. 
Now, let $S$ be a closed subset of $(U_i(K))^m$, and suppose $U_i(K)\models 
S(\bar a)$ for some $\bar a \in \Dom{\pi_{U_i}}$.

Let $\bar S$ be the closure of $S$ in $(G(K))^n$. Then $G(K)\models\bar S
(\bar a)$. Since $\pi_G$ is a specialization, $\pi_G(\bar a)\in \bar S$ (i.e. $G(k)
\models \bar S(\pi_G(\bar a))$). Since $\bar a\in U_i(K)$, we know that $\pi_G
(\bar a) = \pi_{U_i}(\bar a)$. Therefore $\pi_{U_i}(\bar a)\in U_i(k)\cap \bar S(k) 
= S(k)$.
 \end{proof}

\begin{Remark}
Let $\alpha\in U_i(K)\cap\Dom{\pi_G}$ with $a:=\pi_G(\alpha)\not\in U_i(k)$. 
So, $\pi_G(\alpha)\in U_j(k)$ for some $j\neq i$. Then $\alpha\in U_j(K)$. Suppose 
not, i.e. suppose $\alpha\not\in U_j(K)$. Then $\alpha_j = 0$. But this is a Zariski 
closed relation, and since $\pi_G$ is a specialization, it must be preserved by $\pi_G$. 
Therefore, $a_i = 0$. However, $a\in U_j(k)$ means exactly that $a_i\neq 0$. Which 
is a contradiction. Hence $\alpha\in U_j(K)$.

Moreover, as $\alpha\in\Dom{\pi_G}$, it will be in the domain of $\pi_{U_j}:U_j(K)\to U_
j(k)$, which is constructed as in the above proof.  
\end{Remark}

As already remarked above, the affine pieces $U_i$ of $G$ are locally 
closed in $\mathbb A_i^n$, and $\mathbb A_i^n$ is isomorphic to the 
affine space $\mathbb A^n$, via $h_i$. Then $U_i$ is isomorphic to a 
locally closed subset $V_i\subseteq\mathbb A^n$ via $h_i$. 

\begin{Proposition}
Let $U_i$ be an affine piece of $G$. Then the specialization $\pi_{U_i}:
U_i(K)\to U_i(k)$ induces a specialization $\pi_{V_i}:V_i(K)\to V_i(k)$. 
\end{Proposition}
\begin{proof}
Define $\pi_{V_i} := h_i\circ\pi_{U_i}\circ h_i^{-1}$, and let $\Dom{\pi_{V_i}} = 
h_i(\Dom{\pi_{U_i}})$. First, let us verify that $\pi_{V_i}$ is identity 
on $V_i(k)$. Let $a\in V_i(k)$, then $h^{-1}(a)\in U_i(k)$. Since $\pi_{U_i}$ 
is identity on $U_i(k)$, we get $\pi_{U_i}(h^{-1}_i(a)) = h^{-1}_i(a)$. Then 
\[\pi_{V_i}(a) = h_i\circ\pi_{U_i}\circ h_i^{-1}(a) = h_i(\pi_{U_i}(h^{-1}_i(a))) = 
h_i(h^{-1}_i(a)) = a\]

Next, let $S$ be a closed subset of $(V_i(K))^n$, and suppose $V_i(K)\models 
S(\bar a)$ for some $\bar a\in\Dom{\pi_{V_i}}$. Let $T:= h^{-1}_i(S)$. Since 
$h_i^{-1}$ is a homeomorphism, $T$ is closed in $(U_i(K))^n$. Then, since 
$\pi_{U_i}$ is a specialization, $\pi_{U_i}(h_i^{-1}(\bar a))\in T$. Then $\pi_{V_i}(\bar a) 
= h_i(\pi_{U_i}(h_i^{-1}(\bar a)))\in h_i(T) = S$. Hence $\pi_{V_i}$ is a specialization.
\end{proof}

We can repeat the argument we used to construct the specializations 
$\pi_i$ in Proposition~\ref{prop:component-spcl}. Consider the coordinate 
projections $\pr_j:V_i(k)\to k$ for each $j = 1,\ldots, n$. Then $\pr_j(V_i(k)) 
= k\setminus X_{ij}$ for some finite $X_{ij}\subset k$. Define $\pi_{{V_i}_j}:
K\setminus X_{ij}\to k\setminus X_{ij}$ as before: First we define $\Dom{\pi_{{V_i}_j}} 
:= \pr_j(\Dom{\pi_{V_i}})$. Then for any $x\in\Dom{\pi_{{V_i}_j}}$, to define 
$\pi_{{V_i}_j}(x)$ pick an element $(x_1,\ldots, x_n)\in \pr_j^{-1}(x)\cap
\Dom{\pi_{V_i}}$, and define $\pi_{{V_i}_j}(x) := \pr_j(\pi_{V_i}(x_1,\ldots,x_n))$.

As proved in Proposition~\ref{prop:component-spcl}, $\pi_{{V_i}_j}:K\setminus 
X_{ij}\to k\setminus X_{ij}$ is a specialization for each $j=1,\ldots, n$. Also, for 
each $j,k =1,\ldots,n$, the specializations $\pi_{{V_i}_j}$ and $\pi_{{V_i}_k}$ are 
equal on the intersection $\Dom{\pi_{{V_i}_j}}\cap\Dom{\pi_{{V_i}_k}}$ of their 
domains (Proposition~\ref{prop:intrsction-domains}). Moreover there is a common 
extension $\pi_{iK}^0:D_i\to k\setminus X_i$ of the specializations $\pi_{{V_i}_1},
\ldots,\pi_{{V_i}_n}$ where $D_i = \displaystyle\bigcup_{j=1}^n \Dom{\pi_{{V_i}_j}}$ 
and $X_i = \displaystyle\bigcap_{j=1}^n X_{ij}$ (Proposition~\ref{prop:common-ext}). 
Since $X_i\subseteq k$, and since any specialization $K\to k$ must be identity 
on $k$, we can extend $\pi_{iK}^0$ to a specialization $\pi_{iK}^0:K\to k$ with 
$\Dom{\pi_{iK}^0} = D_i\cup X_i$. Therefore, for each $i$ we have a specialization 
$\pi_{iK}^0:K\to k$.

The proposition below immediately follows from the Proposition~\ref{prop:spcl-ind-by-extension}. 

\begin{Proposition}\label{prop:spcl-ext-var}
The specialization $\pi_{V_i}^0:V_i(K)\to V_i(k)$ induced by $\pi_{iK}^0:K\to k$ 
is an extension of the specialization $\pi_{V_i}:V_i(K)\to V_i(k)$ for all $i$, and 
$\pi_{V_i}^0 = \pi_{V_i}$ on $\Dom{\pi_{V_i}}$. 
\end{Proposition} 

\begin{Remark}\label{rmk:special-var-final}
By Proposition~\ref{prop:spcl-ext-var}, we see $\pi_{V_i}(x_1,\ldots, x_n) 
= (x_1^{\pi_{iK}^0}, \ldots, x_n^{\pi_{iK}^0})$ for any $(x_1,\ldots, x_m)\in
\Dom{\pi_{V_i}}$, for each $i$. By definition $\pi_{V_i}(x_1,\ldots, x_n) = h_i\circ\pi_{U_i}\circ 
h_i^{-1}(x_1,\ldots, x_n)$. Write $h_i^{-1}(x_1,\ldots, x_n)$ in homogeneous 
coordinates as $(z_0:\ldots:1:\ldots:z_n)$ where $z_0=x_1,\ldots, z_{i-1}=x_i,
z_{i+1}=x_{i+1},\ldots z_n=x_n$ and $z_i=1$. 

Therefore 
\[(x_1^{\pi_{iK}^0}, \ldots, x_n^{\pi_{iK}^0}) = h_i\circ\pi_{U_i}\circ h_i^{-1}
(x_1,\ldots, x_n) = h_i\circ\pi_{U_i}(z_0:\ldots:1:\ldots:z_n).\]  
Then 
\[\pi_{U_i}(z_0:\ldots:1:\ldots:z_n) = h_i^{-1}(x_1^{\pi_{iK}^0}, \ldots, 
x_n^{\pi_{iK}^0}).\] 
Write $h_i^{-1}(x_1^{\pi_{iK}^0}, \ldots, x_n^{\pi_{iK}^0})$ 
in homogeneous coordinates as $(x_1^{\pi_{iK}^0}: \ldots:1: \dots: x_n^{\pi_{iK}^0})$. 
Moreover, $(x_1^{\pi_{iK}^0}: \ldots:1: \dots: x_n^{\pi_{iK}^0}) = (z_0^{\pi_{iK}^0}:
\ldots:1:\ldots:z_n^{\pi_{iK}^0})$. Therefore 
\[\pi_{U_i}(z_0:\ldots:1:\ldots:z_n)= (z_0^{\pi_{iK}^0}:\ldots:1:\ldots:z_n^{\pi_{iK}^0})\]

If $\pi_{iK}^0:K\to k$ is not maximal, we can extend it to a maximal specialisation 
$\pi_{iK}:K\to k$ as before. Therefore, we can say that on each affine piece $U_i$ 
of $G$, the specialisation $\pi_{U_i}$ induced by the specialisation $\pi_G:G(K)\to 
G(k)$, can be given in terms of a maximal specialisation $\pi_{iK}:K\to k$.

By construction, $\pi_G=\displaystyle\bigcup_{i=0}^n\pi_{U_i}$. Hence, the specialisation 
$\pi_G$ can be given in terms of maximal specialisations $\pi_{iK}$ for $i=0,\ldots,n$ 
on the corresponding affine pieces. 
\end{Remark}


\end{document}